\documentclass[a4,12pt]{amsart}
\usepackage{amssymb,amstext,amsmath,amscd,amsthm,amsfonts,enumerate,graphicx,latexsym,color,stmaryrd}
\usepackage[all]{xy}
\usepackage{comment}

\newcommand{\old}[1]{{\color{red} #1}}

\newtheorem{thm}{Theorem}[section]
\newtheorem*{thm*}{Theorem}

\newtheorem{cor}[thm]{Corollary}
\newtheorem{lem}[thm]{Lemma}
\newtheorem{prop}[thm]{Proposition}
\theoremstyle{definition}

\newtheorem{dfn}[thm]{Definition}
\newtheorem*{dfn*}{Definition}
\newtheorem{rem}[thm]{Remark}
\newtheorem{ques}[thm]{Question}

\newtheorem*{conj*}{Conjecture}
\newtheorem{ex}[thm]{Example}

\theoremstyle{remark}
\newtheorem*{ac}{Acknowledgments}
\newtheorem{claim}{Claim}
\newtheorem*{claim*}{Claim}
\renewcommand{\qedsymbol}{$\blacksquare$}
\numberwithin{equation}{thm}
\def\fa{\mathfrak{a}}

\def\fm{\mathfrak{m}}
\def\fp{\mathfrak{p}}

\def\sX{\mathsf{X}}

\def\rV{\mathrm{V}}

\def\ann{\mathrm{ann}}

\def\Ass{\operatorname{\mathrm{Ass}}}
\def\cid{\operatorname{\mathrm{CI-dim}}}
\def\codim{\operatorname{\mathrm{codim}}}
\def\cx{\operatorname{\mathrm{cx}}}
\def\depth{\operatorname{\mathrm{depth}}}
\def\Ext{\mathrm{Ext}}

\def\grade{\operatorname{\mathrm{grade}}}

\def\height{\operatorname{\mathrm{ht}}}

\def\pd{\operatorname{\mathrm{pd}}}
\def\Spec{\operatorname{\mathrm{Spec}}}
\def\Supp{\operatorname{\mathrm{Supp}}}

\def\tf{\bot}
\def\Tor{\mathrm{Tor}}

\def\ts{\top}

\def\ul{\underline}

\begin{document}
\title{Serre's condition for tensor products and $n$-Tor-rigidity of modules}
\author{Hiroki Matsui}
\address{Graduate School of Mathematical Sciences\\ University of Tokyo, 3-8-1 Komaba, Meguro-ku, Tokyo 153-8914, Japan}
\email{mhiroki@ms.u-tokyo.ac.jp}
\thanks{The author was partly supported by JSPS Grant-in-Aid for JSPS Fellows 19J00158.}
\begin{abstract}
In this paper, we study Serre's condition $(S_n)$ for tensor products of modules over a commutative noetherian local ring.
The paper aims to show the following.
Let $M$ and $N$ be finitely generated module over a commutative noetherian local ring $R$, either of which is $(n+1)$-Tor-rigid.
If the tensor product $M \otimes_R N$ satisfies $(S_{n+1})$, then under some assumptions  $\Tor_{i}^R(M, N) = 0$ for all $i \ge 1$.
The key role is played by $(n+1)$-Tor-rigidity of modules.
As applications, we will show that the result recovers several known results.
\end{abstract}
\maketitle
%%%%%%%%%%%%%%%%%%%%%%%%%%%%%%%%%%%%%%%%%%%%%%%%%%%%%%%%%%%%%%%%%
\section{Introduction}
Throughout this paper, $R$ denotes a commutative noetherian local ring with maximal ideal $\fm$, and $\sX^{n}(R)$ denotes the set of prime ideals $\fp$ of $R$ with $\height \fp \le n$.
 
Torsion in tensor products of finitely generated $R$-modules has been well studied deeply by many authors \cite{Aus, HW1, Jor, Lic} with relation to the Auslander-Reiten conjecture (see \cite{CT}).
Such a study is initiated by Auslander and he proved the following result \cite{Aus, Lic}.

\begin{thm}[Auslander, Lichtenbaum] \label{aus}
Let $R$ be a regular local ring and let $M, N$ be finitely generated $R$-modules.
If $M \otimes_R N$ is torsion-free, then $\Tor_{i}^R(M, N) = 0$ for all $i \ge 1$.
\end{thm}

Three decades later, Huneke and Wiegand \cite{HW1, HW2} generalized the Auslander's result for hypersurface local rings, which is known as the {\it second rigidity thoeorem}.

\begin{thm}[Huneke-Wiegand] \label{hw}
Let $R$ be a hypersurface local ring and let $M, N$ be finitely generated $R$-modules, either of which has rank. %\begin{enumerate}[\rm(1)]
%\item
%$M$ has rank, and
%\item
%$M \otimes_R N$ is reflexive.
%\end{enumerate}
If $M \otimes_R N$ is reflexive, then $\Tor_{i}^R(M, N) = 0$ for all $i \ge 1$.
\end{thm}

Recall that a finitely generated $R$-module $M$ satisfies the {\it Serre's condition $(S_n)$} if the inequality $\depth_{R_\fp} (M_\fp) \ge \min \{n, \height \fp \}$ holds for all $\fp \in \Spec R$.
We note that a non-zero finitely generated $R$-module $M$ is torsion-free if and only if it satisfies $(S_1)$ over a regular local ring and that $M$ is reflexive if and only if it satisfies $(S_2)$ over a hypersurface local ring.
From this observation, Huneke, Jorgensen and Wiegand \cite{HJW} asked the following question.

\begin{ques}[Huneke-Wiegand-Jorgensen]%\cite{HJW}
Let $R$ be a complete intersection local ring of codimension $c$ and $M$, $N$ finitely generated $R$-modules. 
If $M \otimes_R N$ satisfies $(S_{c+1})$, then does $\Tor_{i}^R(M, N) = 0$ hold for all $i \ge 1$?
\end{ques}
\noindent
The answer to this question is given by Dao \cite{Dao} under some extra assumptions on $M$ and $N$:

\begin{thm}[Dao]\label{dao}
Let $R$ be a complete intersection local ring of codimension $c$ whose completion is a quotient of an unramified regular local ring.
Let $M, N$ be finitely generated $R$-modules.
Assume:
\begin{enumerate}[\rm(1)]
\item
$N_\fp$ is free for $\fp \in \sX^c(R)$.
\item	
$M$ and $N$ satisfy $(S_c)$.
\end{enumerate}
If $M \otimes_R N$ satisfies $(S_{c+1})$, then $\Tor_i^R(M, N) = 0$ for all $i \ge 1$.
\end{thm}

The key role of the proof of Theorem \ref{dao} is played by {\it $n$-Tor-rigidity} of modules; see Definition \ref{ntr} for the definition of $n$-Tor-rigid modules.
%Here, we say that a finitely generated $R$-module $M$ is {\it  $n$-Tor-rigid} if for any finitely generated $R$-module $N$, the vanishing of $n$ consecutive Tor-modules of $M$ and $N$ implies the vanishing of all higher Tor-modules.
Indeed, the proof relies upon the fact that the vanishing of {\it eta pairing}, which is a generalization of {\it Hochster's theta pairing}, implies that $c$-Tor-rigidity of a pair of finitely generated $R$-modules.
The unramified assumption on an embedded regular local ring is used here.
In addition, it is shown by Murthy \cite{Mur} that over a complete intersection local ring $R$ of codimension $c$, every finitely generated $R$-module is $(c+1)$-Tor-rigid.
In this paper, we consider another direction of a variant of Theorems \ref{aus} and \ref{hw}.
Namely, we consider the following question.
 
\begin{ques}
Let $R$ be a noetherian local ring and $M, N$ finitely generated $R$-modules and $n$ a non-negative integer.
If $M \otimes_R N $ satisfies $(S_{n+1})$ and either $M$ or $N$ is $(n+1)$-Tor-rigid, then does $\Tor_i^R(M, N) = 0$ hold for all $i \ge 1$?
\end{ques}

The aim of this paper is to give an answer to this question.
Precisely, we prove the following result.

\begin{thm}[see Theorem \ref{2srt}]\label{main}
Let $n$ be a non-negative integer and let $R$ be a noetherian local ring satisfying $(S_n)$ for $n \ge 1$ and $(S_1)$ for $n=0$.
Let $M$ and $N$ be finitely generated $R$-modules.
Assume:
\begin{enumerate}[\rm(1)]
\item
$N$ satisfies $(S_{n})$ and $N_\fp$ is free for $\fp \in \sX^n(R)$.
\item
$M$ is stably isomorphic to the $n$th syzygy of an $(n+1)$-Tor-rigid module.
%\item
%$M \otimes_R N$ satisfies $(S_{n+1})$.
\end{enumerate}
If $M \otimes_R N$ satisfies $(S_{n+1})$, then $\Tor_{i}^R(M, N) = 0$ for all $i \ge 1$.
\end{thm}
\noindent
Here, we say that two finitely generated $R$-modules $M$ and $N$ are {\it stably isomorphic} if there are finitely generated free $R$-modules $F$ and $G$ such that $M \oplus F \cong N \oplus G$.

If we assume $R$ is a complete intersection local ring of codimension $c$, then every finitely generated $R$-module is $(c+1)$-Tor-rigid and that a finitely generated $R$-module is the $c$th syzygy of some finitely generated $R$-module if and only if it satisfies $(S_c)$.
Thus, this result recovers Theorem \ref{dao}.
Moreover, I want to emphasize that the unramified assumption on an embedded regular local ring is removed in our theorem. 

The organization of this paper is as follows.
The Section 2 is devoted for the preparation of the proof.
In Section 3, we will prove Theorem \ref{2srt} and give several applications of the main theorem, which recovers some known results such as Theorem \ref{dao}.

\section{Preliminaries.}
In this section, we recall several basic definitions including that of an $n$-Tor-rigid module for later use.

\if0\begin{dfn}
For a finitely generated $R$-module $M$, define the {\it codimension} of $M$ to be 
$$
\codim M := \inf \{\height \fp \mid \fp \in \Supp M\}.
$$
From the definition, $\codim M \ge n$ if and only if $M_\fp = 0$ for all $\fp \in \sX^n(R)$.
\end{dfn}\fi

\begin{dfn}(\cite{AGP})
Let $M$ be a finitely generated $R$-module.
A {\it (codimension $c$) quasi-deformation} of $R$ is a diagram $R \rightarrow R' \twoheadleftarrow Q$ of local rings such that $R \rightarrow R'$ is fully faithful and $R' \twoheadleftarrow  Q$ is surjective with kernel generated by a $Q$-regular sequence (of length $c$).
The {\it complete intersection dimension} of $M$ is defined to be
$$
\cid_R (M) := \inf \{\pd_{Q}(M \otimes_R R') - \pd_{Q} R' \mid R \rightarrow R' \twoheadleftarrow Q \mbox{ is a quasi-deformation of } M \}.
$$
\end{dfn}

\begin{dfn}
Let $M$ be a finitely generated $R$-module. 	
The {\it complexity} of $M$ is
$$
\cx_R(M) := \inf\{ c \ge 0 \mid \mbox{there exists $r >0$ such that } \beta_i(M) \le r i^{c-1} \mbox{ for all } i \gg 0 \}.
$$
Here, $\beta_i(M) := \dim_k \Ext_R^i(M, k)$ denotes the $i$th {\it Betti number} of $M$.
\end{dfn}

For the basic properties of complete intersection dimension and complexity, we refer the reader to \cite{Av, AB, AGP}.
We only record the following important fact on complete intersection dimension and complexity.

\begin{prop}[{\cite[Theorem 1.3]{AGP} and \cite[Theorem 2.3]{Gul}}]
The following are equivalent for a noetherian local ring $R$.
\begin{enumerate}[\rm(1)]
\item
$R$ is complete intersection.
\item
$\cid_R (M)< \infty$ for every finitely generated $R$-module $M$.
\item
$\cx_R (M) < \infty$ for every finitely generated $R$-module $M$.
\end{enumerate}
This is the case, $\cx_R(M) \le \codim R$ for every finitely generated $R$-module $M$.
\end{prop}

Here, we give the definition of an $n$-Tor-rigid module.

\begin{dfn}\label{ntr}
Let $n$ be a positive integer.
We say that a finitely generated $R$-module $M$ is {\it $n$-Tor-rigid} if for a finitely generated $R$-module $N$ and a positive integer $t$, the $n$-consecutive vanishing $\Tor_i^R(M, N) = 0$ for $i = t, t+1, \ldots, t+n-1$ implies that $\Tor_i^R(M, N) = 0$ for all $i \ge t$.%the vanishing of $n$ consecutive Tor's implies the vanishing of all higher Tor's for each $R$-module $N$.
\end{dfn}

There are various examples of n-Tor-rigid modules in the literature.

\begin{ex}\label{ex}
\begin{enumerate}[\rm(a)]
\item
(\cite[Corollary 1.9]{Mur}) 
Assume that $R$ is a complete intersection local ring of codimension $c$.
Then every finitely generated $R$-module is $(c+1)$-Tor-rigid.
%($R$ doesn't need to be a quotient of an unramified regular local ring since the proof is reduced to Tor-rigidity of modules over regular local ring, which is proved by Lichitenbaum in general.)
%Actually, Murthy proved the statement when the completion of $R$ is a quotient of an unramified regular local ring because this relies on  
\item
(\cite[Corollary 6.8]{Dao})
Let $R$ be a local complete intersection ring of codimension $c>0$ whose completion is a quotient of an unramified regular local ring.
Then a finitely generated $R$-module $M$ is $c$-Tor-rigid if $\cx_R(M) < c$.
\item
(cf. \cite[Corollary 4.3]{CD} and \cite[Proposition 2.3]{Jor1})
Let $M$ be a finitely generated $R$-module.
If $\cid_R(M) = 0$ and $\cx_R(M) = c$, then $M$ is $(c+1)$-Tor-rigid.
\item
(\cite[Theorem 5(ii)]{Bur})
Let $I$ be an ideal satisfying $\fm I \neq \fm (I : \fm)$.
Then $R/I$ and hence $I$ is $2$-Tor-rigid.
\item
(\cite[Lemma, page 316]{LV})
Let $M$ be an $R$-module with $\fm M \neq 0$.
Then $\fm M$ is $2$-Tor-rigid.
\end{enumerate}
\end{ex}

We introduce the following fact which is essentially shown by Auslander \cite[Lemma 3.1]{Aus}, since it is the prototypical result of our main theorem.

\begin{prop}\label{ausprop}
Let $R$ be a noetherian local ring and $M, N$ finitely generated $R$-modules.
Assume:
\begin{enumerate}[\rm(1)]
\item
$N$ has a (constant) rank.
\item	
$M$ is $1$-Tor-rigid.
%\item
%$M \otimes_R N$ satisfies $(S_1)$.
\end{enumerate}
If $M \otimes_R N$ satisfies $(S_1)$, then $\Tor_i^R(M, N) = 0$ for all $i \ge 1$.	
\end{prop}

\begin{proof}
Consider the short exact sequence $0 \to \ts N \to N \to \tf N \to 0$ where $\ts N$ and $\tf N$ are torsion and torsion-free part of $N$, respectively.	
Tensoring this sequence with $M$, we obtain an exact sequence
$$
\Tor_1^R(M, \tf N) \to M \otimes_R \ts N \xrightarrow{\alpha} M \otimes_R N \xrightarrow{\beta} M \otimes_R \tf N \to 0.
$$ 
As $M \otimes_R \ts N$ is torsion and $M \otimes_R N$ is torsion-free, the map $\alpha$ must be zero and thus $\beta$ is an isomorphism.
If it is shown that $\Tor_1^R(M, \tf N) = 0$, then $M \otimes_R \ts N =0$ and hence $\ts N = 0$.
Therefore we may assume that $N$ is torsion-free.

Since $N$ is a torsion-free $R$-module having a rank, we can take a short exact sequence
$$
0 \to N \to F \to C \to 0
$$
with $F$ free and $C$ torsion.
From this short exact sequence, we get an exact sequence
$$
0 \to \Tor_1^R(M, C) \to M \otimes_R N \to M \otimes_R F \to M \otimes_R C \to 0.
$$ 
Since $\Tor_1^R(M, C)$ is torsion and $M \otimes_R N$ is torsion-free, we see that $\Tor_1^R(M, C) = 0$.
Then the $1$-Tor-rigidity of $N$ shows $\Tor_i^R(M, C) = 0$ for all $i \ge 1$ and in particular $\Tor_i^R(M, N) = 0$ for all $i \ge 1$.
\end{proof}

\section{Main theorem and its applications}

The purpose of this section is to prove our main theorem.
Before beginning the proof, we establish several lemmas.

\begin{lem}\label{lem1}
Let $M$ and $N$ be finitely generated $R$-modules, and $\underline{x}:= x_1, \ldots, x_n$ an $N$-regular sequence.
If $\Tor_1^R(M, N/\underline{x}N) = 0$, then $\Tor_1^R(M, N) = 0$.
The converse holds if moreover $\ul{x}$ is regular on $M \otimes_R N$. 
\end{lem}

\begin{proof}
%Assume $\Tor_1^R(M, N/\ul{x}N) = 0$.
%From the short exact sequence $0 \to N \xrightarrow{x_1} N \to N/x_1N \to 0$, we get an exact sequence
%$$
%\Tor_1^R(M, N) \xrightarrow{x_1} \Tor_1^R(M, N) \to \Tor_1^R(M, N/x_1N)=0
%$$	
%By Nakayama's lemma, we conclude that $\Tor_1^R(M, N) =0$.
The first statement follows easily from Nakayama's lemma.

Assume that $\ul{x}$ is regular on both $M \otimes_R N$ and $N$ and that $\Tor_1^R(M, N) = 0$.
From the short exact sequence $0 \to N \xrightarrow{x_1} N \to N/ x_1 N \to 0$, we obtain an exact sequence
$$
\Tor_1^R(M, N) = 0 \to \Tor_1^R(M, N/x_1 N) \to M \otimes_R N \xrightarrow{x_1} M \otimes_R N \to M \otimes_R (N/ x_1 N) \to 0.
$$
Then we have $\Tor_1^R(M, N/x_1 N) = 0$, since $x_1$ is regular on $M \otimes_R N$.
On the other hand, the isomorphism $M \otimes_R (N/x_1 N) \cong (M \otimes_R N) /x_1 (M \otimes_R N)$ shows that the sequence $x_2, \ldots, x_n$ is regular on both $M \otimes_R (N/x_1 N)$ and $N/x_1N$.
Therefore the induction argument on $n$ shows that $\Tor_1^R(M, N /\ul{x} N) = 0$.  
\end{proof}

\begin{lem}\label{van}
Let $M$ and $N$ be finitely generated $R$-modules, and $\underline{x}:= x_1, \ldots, x_n$ an $N$-regular sequence such that $\ul{x} \Tor_{i}^R(M, N) = 0$ for $i = 1, \ldots, n$.
If $\Tor_{n+1}^R(M, N/\ul{x}^{2^{n-1}}N)=0$, then $\Tor_i^R(M, N) = 0$ for $i = 1,2, \ldots, n+1$. 
Here, $\ul{x}^{2^{n-1}}$ denotes the sequence $x_1^{2^{n-1}}, \ldots, x_n^{2^{n-1}}$.
\end{lem}

\begin{proof}
Notice that	$\Tor_{n+1}^R(M, N) = 0$ holds by Lemma \ref{lem1}.
We prove $\Tor_i^R(M, N) = 0$ for $i = 1, \ldots, n+1$ by induction on $n$.

Suppose that $n=1$.
Consider the following pullback diagram of $N \xrightarrow{x_1} N$ and a free cover $F \to N$:
\if0$$
\begin{CD}
@. @. 0 @. 0 @. \\
@. @. @AAA @AAA @. \\
@. @. N/x_1N @= N/x_1N @. \\
@. @. @AAA @AAA @. \\
0 @>>> \Omega N @>>> F @>>> N @>>> 0 \\
@. @| @AAA @AAx_1A @. \\
0 @>>> \Omega N @>>> \Omega(N/x_1N) @>>> N @>>> 0 \\
@. @. @AAA @AAA @. \\
@. @. 0 @. 0 @. 
\end{CD}
$$\fi
$$
\xymatrix{
& & 0 & 0 \\
& & N/x_1N \ar@{=}[r]\ar[u]& N/x_1N \ar[u]\\
0 \ar[r] & \Omega N \ar[r] \ar@{=}[d] & F \ar[r] \ar[u]& N \ar[r] \ar[u]& 0 \\ 
0 \ar[r] & \Omega N \ar[r] & W \ar[r] \ar[u] & N \ar[r] \ar[u]_{x_1} & 0 \\
& & 0 \ar[u] & 0 \ar[u]
}
$$
Then the middle column means that $W$ is the first syzygy of $N/x_1 N$ i.e., $W = \Omega(N/x_1 N)$.
Applying $M \otimes_R -$ to the middle and the bottom rows, we get a commutative diagram
\if0$$
\begin{CD}
0 @>>> \Tor_1^R(M, N) @>>> M \otimes_R \Omega N \\
@. @AAx_1A @| \\
0 @>>> \Tor_1^R(M, N) @>>> M \otimes_R \Omega N  
\end{CD}
$$\fi
$$
\xymatrix{
\Tor_1^R(M, F) =0 \ar[r] & \Tor_1^R(M, N) \ar[r] & M \otimes_R \Omega N \ar@{=}[d] \\ 
\Tor_1^R(M, \Omega(N/x_1N)) =0 \ar[u] \ar[r] & \Tor_1^R(M, N) \ar[r] \ar[u]^{x_1} & M \otimes_R \Omega N.
}
$$
Thus the assumption $x_1 \Tor_1^R(M, N) = 0$ implies that $\Tor_1^R(M, N) = 0$.

Now let $n \ge 1$ and set $N_{k}:= N/(x_1^{2^{n-1}}, \ldots, x_k^{2^{n-1}})N$ for $k = 0, \ldots, n$.	
First we show:
\begin{claim*}
$\ul{x}^{2^{k}} \Tor_i^R(M, N_k) = 0$ for $k+1 \le i \le n$.
\end{claim*}

\begin{proof}
Let us proceed by induction on $k$.
If $k=0$, there is nothing to prove.
Consider the case of $k \ge 1$ and assume $\ul{x}^{2^{k-1}} \Tor_i^R(M, N_{k-1}) = 0$ for $k \le i \le n$.	
The short exact sequence $0 \to N_{k-1} \xrightarrow{x_k^{2^{n-1}}} N_{k-1} \to N_{k} \to 0$ induces the following commutative diagram
\if0$$
\begin{CD}
\Tor_i^R(M, N_{k-1}) @>>> \Tor_{i}^R(M, N_{k}) @>>> \Tor_{i-1}(M, N_{k-1}) \\
@AAx_j^{2^{k-1}} = 0A @AAx_j^{2^{k-1}}A @AAx_j^{2^{k-1}}=0A  \\ 
\Tor_i^R(M, N_{k-1}) @>>> \Tor_{i}^R(M, N_{k}) @>>> \Tor_{i-1}(M, N_{k-1})
\end{CD}
$$\fi
$$
\xymatrix{
\Tor_i^R(M, N_{k-1}) \ar[r] & \Tor_{i}^R(M, N_{k}) \ar[r] & \Tor_{i-1}(M, N_{k-1}) \\
\Tor_i^R(M, N_{k-1}) \ar[r] \ar[u]_{x_j^{2^{k-1}} = 0} & \Tor_{i}^R(M, N_{k}) \ar[r] \ar[u]_{x_j^{2^{k-1}}} & \Tor_{i-1}(M, N_{k-1}) \ar[u]_{x_j^{2^{k-1}} = 0}
}
$$
with exact rows.
By the induction hypothesis, the left and the right vertical maps are zero.
Hence we obtain $\ul{x}^{2^k} \Tor_i(M, N_{k}) =0$ for $k+1 \le i \le n$.
\renewcommand{\qedsymbol}{$\square$}
\end{proof}
\noindent
In particular, we have $\ul{x}^2 \Tor_i^R(\Omega M, N_1) =0$ for $1 \le i \le n-1 $.
On the other hand, there is an isomorphism
$$
\Tor_n^R(\Omega M, N_1 /(x_2^{2^{n-1}}, \ldots, x_{n}^{2^{n-1}}) N_1) \cong \Tor_{n+1}^R(M, N/\ul{x}^{2^{n-1}} N) =0
$$
from the assumption.
Therefore, the induction hypothesis shows $\Tor_i^R(\Omega M, N_1) =0$ for $1 \le i \le n$.
By Lemma \ref{lem1}, we get $\Tor_i^R(M, N) =0$ for $2 \le i \le n+1$.
Finally, from $\Tor_2^R(M, N_1) =0$ and $x_1^{2} \Tor_i^R(M, N) = 0$ for $i=1,2$, we also obtain $\Tor_1^R(M, N) = 0$.
\if0Using the same argument as above, we have a commutative diagram
$$
\begin{CD}
0 @>>> \Tor_1^R(M, N) @>>> M \otimes_R \Omega N \\
@. @AAx_1^{2^{n-1}}A @| \\
\Tor_1^R(M, \Omega N_1) @>>> \Tor_1^R(M, N) @>>> M \otimes_R \Omega N.  
\end{CD}
$$
As we have already shown that $\Tor_1(M, \Omega N_1) \cong \Tor_1(\Omega M, N_1) =0$, we get $\Tor_1^R(M, N) =0$.\fi
\end{proof}

Recall that the {\it codimension} of a finitely generated $R$-module $M$ is defined to be
$$
\codim M := \inf \{\height \fp \mid \fp \in \Supp M\}.
$$
From the definition, $\codim M \ge n$ if and only if $M_\fp = 0$ for all $\fp \in \sX^n(R)$.

\begin{lem}\label{codim}
Let $M$ and $N$ be finitely generated $R$-modules such that $\cid_{R_\fp} (N_\fp) =0$ for all $\fp \in \sX^n(R)$.
If $\codim \Tor_{i}^R(M, N) \ge n+1$ for all $i \gg 0$, then $\codim \Tor_{i}^R(M, N) \ge n+1$ for all $i \ge 1$.
\end{lem}

\begin{proof}
Fix $\fp \in \sX^n(R)$ and it suffices to show that $\Tor_{i}^R(M, N)_\fp = 0$ for all $i \ge 1$.
From the assumption we get $\Tor_{i}^R(M, N)_\fp = 0$ for all $i \gg 0$ and this implies $\Tor_{i}^R(M, N)_\fp = 0$ for all $i \ge 1$ as $\cid_{R_\fp} (N_\fp) = 0$; see \cite[Theorem 4.9]{AB}.
\end{proof}

Now we are ready to show our main theorem.
The idea of the proof has been already appeared in the proof of Proposition \ref{ausprop}.

\begin{thm}\label{2srt} 
%Let $n$ be a positive integer, $R$ a noetherian local ring satisfying $(S_n)$, and $M$, $N$ be non-zero finitely generated $R$-modules.
Let $n$ be a non-negative integer, and let $R$ be a noetherian local ring satisfying $(S_n)$ if $n \ge 1$ and $(S_1)$ if $n=0$.
Let $M$ and $N$ be finitely generated $R$-modules.
Assume:
\begin{enumerate}[\rm(1)]
\item
$N$ satisfies $(S_{n})$ and $\cid_{R_\fp} (N_\fp) <\infty$ for $\fp \in \sX^n(R)$.
\item
$M$ is stably isomorphic to the $n$th syzygy of an $(n+1)$-Tor-rigid module $M'$,
\item
$\codim \Tor^R_{i}(M, N) \ge n+1$ for all $i \gg 0$,
%\item
%$M \otimes_R N$ satisfies $(S_{n+1})$.
\end{enumerate}
If $M \otimes_R N$ satisfies $(S_{n+1})$, then $\Tor_{i}^R(M, N) = 0$ for all $i \ge 1$.
\end{thm}

\begin{proof}%[Proof of Theorem \ref{2srt}]
If $n \ge \dim R$, there is nothing to prove by the assumption (3).
We may assume that $n < \dim R$.

First consider the case of $n=0$.
Using the same argument as in the proof of Proposition \ref{ausprop}, we may assume $N$ is torsion-free and hence satisfies $(S_1)$ because $R$ satisfies $(S_1)$. 
It follows from \cite[Proposition 2.4]{DS} and the assumption (1) that there is a short exact sequence $0 \to N \to F \to C \to 0$ with $F$ free.
Then the same proof as in Proposition \ref{ausprop} works and we complete the proof in this case. 

Next let us consider the case of $n \ge 1$.
Again using \cite[Proposition 2.4]{DS}, there is a short exact sequence
\begin{align} \label{lfa}
0 \to N \to F \to C \to 0 \tag{$*$}
\end{align}
with $F$ free, $C$ satisfies $(S_{n-1})$ and $\Ext^1_R(C, R) = 0$.
Since $N$ satisfies $(S_n)$, $\cid_{R_\fp}(N_\fp) = 0$ for all $\fp \in \sX^n(R)$.
Thus, $\Ext_{R_\fp}^1(C_\fp, R_\fp) = 0$ shows that $\cid_{R_\fp}(C_\fp) = 0$ for all $\fp \in \sX^n(R)$ by \cite[Lemma 1.1.10]{Chr} and \cite[Theorem 1.4]{AGP}.
From Lemma \ref{codim}, we have $\codim \Tor_{i}(M', C) \ge n+1$ for all $i \ge 1$.
Tensoring $M$ with the above sequence (\ref{lfa}), we get an exact sequence
$$
0 \to \Tor_1^R(M, C) \to M \otimes_R N \to M \otimes_R F \to M \otimes_R C \to 0.
$$
Since $M\otimes_R N$ satisfies $(S_{n+1})$, and $\Tor_1^R(M, C) \cong \Tor_{n+1}^R(M', C)$ is torsion, we get $\Tor_1^R(M, C) = 0$.

\if0\old{\begin{claim}
Set $\fa:= \bigcap_{1 \le i \le n} \ann \Tor_{i}^R(M', C)$.
Then
$$
\grade(\fa, C) \ge n-1 \mbox{ and }\grade(\fa, M \otimes_R C) \ge n.
$$
\end{claim}

\begin{proof}[Proof of Claim 1]
It follows from $\codim \Tor_{>0}(M', C) \ge n+1$ that $V(\fa)$ consists only prime ideals of height at least $n+1$.
Since $C$ is ($n-1)$-torsionfree and in paricular satisfies $(S_{n-1})$, one has
\begin{align*}
\grade(\fa, C) &= \inf \{\depth_{R_\fp} C_\fp \mid \fp \in V(\fa)\} \\
&\ge \inf \{\depth_{R_\fp} C_\fp \mid \height \fp \ge n+1 \} \ge n-1.
\end{align*}

Fix a prime ideal $\fp$ with height at least $n+1$.
If $\height \fp \ge n+1$, by the depth lemma, the exact sequence
$$
0 \to (M \otimes_R N)_\fp \to (M \otimes_R F)_\fp \to (M \otimes_R C)_\fp \to 0
$$
implies that $\depth_{R_\fp} (M \otimes_R C)_\fp \ge n$.
Thus,
\begin{align*}
\grade(\fa, M \otimes_R C) &= \inf \{\depth_{R_\fp} (M \otimes_R C)_\fp \mid \fp \in V(\fa)\} \\
&\ge \inf \{\depth_{R_\fp} (M \otimes_R C)_\fp \mid \height \fp \ge n+1\} \ge n.
\end{align*}
\end{proof}}
\fi

Next, we prove that $\Tor_i^R(M', C) = 0$ for $i = 2, \ldots, n+1$.
To this end, we may assume $n \ge 2$.
Set $\fa:= \bigcap_{1 \le i \le n} \operatorname{\mathrm{ann}} \Tor_{i}^R(M', C)$.
It follows from $\codim \Tor_{i}^R(M', C) \ge n+1$ for all $i \ge 1$ that $\rV(\fa)$ consists only prime ideals of height at least $n+1$.
Using $(S_{n-1})$ condition on $C$, one has the following (in)equalities:
\begin{align*}
\grade(\fa, C) &= \inf \{\depth_{R_\fp} C_\fp \mid \fp \in V(\fa)\} \ge \inf \{\depth_{R_\fp} C_\fp \mid \height \fp \ge n+1 \} \ge n-1.
\end{align*}
Here, the first equality uses \cite[Proposition 1.2.10 (a)]{BH}.
On the other hand, for a prime ideal $\fp$ with height at least $n+1$, the depth lemma applying to the short exact sequence
$$
0 \to (M \otimes_R N)_\fp \to (M \otimes_R F)_\fp \to (M \otimes_R C)_\fp \to 0
$$
implies that $\depth_{R_\fp} (M \otimes_R C)_\fp \ge n$.
Thus the similar argument as above shows that $\grade (\fa, M \otimes_R C) \ge n$.
Therefore we can take a sequence $\underline{x}:= x_1, \ldots, x_{n-1}$ of elements from $\fa$ which is regular on both $C$ and $M \otimes_R C$.
Then we get $\Tor_{n}^R(\Omega M', C/ \ul{x}^{2^{n-2}} C) \cong \Tor_1^R(M, C/ \ul{x}^{2^{n-2}} C) = 0$ by Lemma \ref{lem1}.
%This also shows the vanishing $\Tor_1^R(\Omega M', \Omega^{n-1} (C/ \ul{x}^{2^{n-1}} C))= 0$.
Because $\ul{x}$ is taken from the annihilators of $\Tor_i^R(\Omega M', C) \cong \Tor_{i+1}^R(M', C)$ for $ 1 \le i \le n-1$, Lemma \ref{van} gives us that $\Tor_i^R(M', C) \cong \Tor_{i-1}^R(\Omega M', C) = 0$ for $2 \le i \le n+1$.
%Therefore, $\Tor_i^R(M', N)=0$ for $i = 1, \ldots, n$.

\if0\begin{claim*}\label{claim3}
$\Tor_1^R(M, C/\underline{x}^{2^{n-2}}C) = 0$.
\end{claim*}

\begin{proof}[Proof of Claim]
The short exact sequence $0 \to C \xrightarrow{x_1^{2^{n-2}}} C \to C/x_1^{2^{n-2}} C \to 0$ yields
$$
\Tor_1^R(M, C) = 0 \to \Tor_1^R(M, C/x_1^{2^{n-2}} C) \to M \otimes_R C \xrightarrow{x_1^{2^{n-2}}} M \otimes_R C \to M \otimes_R C/x_1^{2^{n-2}} C \to 0.
$$ 
Hence $\Tor_1^R(M, C/x_1^{2^{n-2}} C) = 0$ and $M \otimes_R C/x_1^{2^{n-2}} C \cong (M \otimes_R C)/x_1^{2^{n-2}}(M \otimes_R C)$.
In particular, $x_2^{2^{n-2}}$ is regular on $M \otimes_R C/x_1^{2^{n-2}} C$.
Thus, the same argument as above works for $C/x_1^{2^{n-2}} C$ and we are done by induction.
\end{proof}\fi

To use the $(n+1)$-Tor-rigidity of $M'$, it remains to show $\Tor_1^R(M', C) = 0$.
If this is true, the vanishing $\Tor_i^R(M', C) = 0$ for $ 1 \le i \le n+1$ implies that $\Tor_{i}^R(M', C) = 0$ for all $i \ge 1$ and hence we conclude that $\Tor_i^R(M, N) = 0$ for all $i \ge 1$.
Assume the contrary that $\Tor_1^R(M', C) \neq 0$ and take $\fp \in \Ass \Tor_1^R(M', C)$.
Notice that $\height \fp \ge n+1$ as $\codim \Tor_1^R(M', C) \ge n+1$.
Let $F_\cdot$ be a free resolution of $M'$ and decompose it into short exact sequences
%\begin{align*}
%&0 \to \Omega M' \to F_0 \to M' \to 0 \\
%&0 \to \Omega^2 M' \to F_1 \to \Omega M' \to 0 \\
%&\cdots \\
%&0 \to M \to F_{n-1} \to \Omega^{n-1} M' \to 0 \\
%\end{align*}
$
0 \to \Omega^{k+1}M' \to F_k \to \Omega^k M' \to 0
$
for $k \ge 0$.
Tensoring these sequences with $C$, we obtain the following exact sequences
\begin{align*}
0 \to \Tor_1^R(M', C) \to \Omega M' \otimes_R C \to &F_0 \otimes_R C \to M' \otimes_R C \to 0 \\
0 \to \Omega^2 M' \otimes_R C \to &F_1 \otimes_R C \to \Omega M' \otimes_R C \to 0 \\
&\cdots \\
0 \to M \otimes_R C \to &F_{n-1} \otimes_R C \to \Omega^{n-1} M' \otimes_R C \to 0.
\end{align*}
Here we use $\Tor_1^R(\Omega^{i} M', C) \cong \Tor_{i+1}^R(M', C) = 0$ for $1 \le i \le n$.
%Since $\height \fp \ge n+1$, it has been already mentioned in the proof of Claim \ref{2}, $\depth_{R_\fp} (M \otimes_R C)_\fp \ge n$.
It has been already explained above that $\depth_{R_\fp} (M \otimes_R C)_\fp \ge n$ because $\height \fp \ge n+1$.
Using the depth lemma to the above sequences, we obtain
$
\depth_{R_\fp} (\Omega M' \otimes_R C)_\fp \ge 1
$
This contradicts to $\fp \in \Ass \Tor_1^R(M', C)$.
Thus we conclude that $\Tor_1^R(M', C) =0$ and we are done.
\end{proof}

\if0
\begin{rem}
If $R$ satisfies $(S_1)$, then Theorem \ref{2srt} also holds for $n=0$.

Indeed, since $M \otimes_R N$ satisfies $(S_1)$, we may assume $N$ satisfies $(S_1)$ by replacing $N$ with its torsion-free part; see the proof of Proposition \ref{ausprop}.
Then $N$ is embeded into a free $R$-module and the same argument works.
\end{rem}\fi

%\begin{rem}
%In the above proof, the assumption (1) is only used to find a short exact sequence (\ref{lfa}) and to apply Lemma \ref{codim} for the pair $(M' , C)$.
%Thus, 
%\end{rem}

\begin{ex}
Let $k$ be a field, $R := k[[x,y]]/(xy)$ and $M:= R/(x)$.
Note that $M$ is maximal Cohen-Macaulay having finite complete intersection dimension.
Moreover $M$ is the first syzygy of an $2$-Tor-rigid module $R/(y)$. 
Then $M \otimes_R M \cong M$ satisfies $(S_2)$ but $\Tor_1^R(M, M) \cong k \neq 0$.
This explains that the consition (3) in Theorem \ref{2srt} is necessary.
Indeed, $\codim \Tor_i^R(M, M) =1$ for all odd $i$. 	
\end{ex}

%\section{Applications}
%In this section, we give several applications of Theorem \ref{2srt} including some known results.
For the rest of this paper, we will give several applications of the main theorem using $n$-Tor-rigid modules appeared in Example \ref{ex}.

%Since every finitely generated $R$-module over a complete inersection local ring of codimension $c$ is $(c+1)$-Tor-rigid, we obtain the following corollary (cf. \cite[Theorem 7.6]{Dao}).
First consider the case of $n=c$ for complete intersection local rings of codimension $c$.
As we have explained in the introduction that this case of Theorem \ref{2srt} recovers Dao's result.

\begin{cor}[{cf. \cite[Theorem 7.6]{Dao}}]\label{srt}
Let $R$ be a complete intersection local ring of codimension $c$ and let $M, N$ be finitely generated $R$-modules.
Assume: 
\begin{enumerate}[\rm(a)]
\item
$M$ and $N$ satisfy $(S_{c})$.
\item
$\codim \Tor_{i}^R(M, N) \ge c+1$ for all $i \gg 0$.
%\item
%$M \otimes_R N$ satisfies $(S_{c+1})$.
\end{enumerate}
If $M \otimes_R N$ satisfies $(S_{c+1})$, then $\Tor_{i}^R(M, N) = 0$ for all $i \ge 1$.
\end{cor}

\begin{proof}
The condition (1) in Theorem \ref{2srt} follows from \cite[Proposition 1.6]{AGP}.
Since $R$ is complete intersection, the assumption (a) shows that $M$ is (stably) isomorphic to the $(c+1)$st syzygy of some finitely generated $R$-module $M'$.	
It follows from Example \ref{ex}(1) that $M'$ is $(c+1)$-Tor-rigid module.
Thus, the condition (2) in Theorem \ref{2srt} is satisfied.
\end{proof}

For the case of $n=c-1$ over a complete intersection local ring, the main theorem also recovers the result by Celikbas.

\begin{cor}[{cf. \cite[Theorem 3.4]{Cel}}]
Let $R$ be a complete intersection local ring of codimension $c \ge 1$ whose completion is a quotient of an unramified regular local ring.
Let $M$ and $N$ be finitely generated $R$-modules.
Assume: 
\begin{enumerate}[\rm(a)]
\item
$M$ and $N$ satisfy $(S_{c-1})$.
\item
$\codim \Tor_{i}^R(M, N) \ge c$ for all $i \gg 0$.
%\item
%$M \otimes_R N$ satisfies $(S_{c})$.
\end{enumerate}
If $M \otimes_R N$ satisfies $(S_{c})$, then $\cx_R(M)= \cx_R(N) =c$ or $\Tor_{i}^R(M, N) = 0$ for all $i \ge 1$.
\end{cor}

\begin{proof} 
If $\cx_R (M) < c$ or $\cx_R(N) < c$, then $M$ or $N$ is $c$-Tor-rigid by \cite[Corollary 6.8]{Dao}.
Thus, using the similar argument as above, we can apply Theorem \ref{2srt} for $n=c-1$.
\end{proof}

\if0
\begin{cor}
Let $R$	be a ring satisfying $(S_n)$ and $M$, $N$ finitely generated $R$-module.
Assume:
\begin{enumerate}[\rm(a)]
%\item
%$N$ satisfies $(S_{n})$, $\cid N = 0$, and $\cx_R(N) \le n$.
%\item
%$M$ satisfies $(S_n)$.
\item
$M$ and $N$ satisfy $(S_n)$.
\item
Either of $M$ or $N$ has complete intersection dimension zero and complexity at most $n$.
\item
$\codim \Tor_{i}(M, N) \ge n+1$ for all $i \ge 1$.
%\item
%$M \otimes_R N$ satisfies $(S_{n+1})$.
\end{enumerate}
If $M \otimes_R N$ satisfies $(S_{n+1})$, then $\Tor_{i}^R(M, N) = 0$ for all $i \ge 1$.
\end{cor}

\begin{proof}
If $$	
\end{proof}
\fi

Recall that an ideal $I$ of $R$ is called {\it Burch} if it satisfies $\fm I \neq \fm(I : \fm)$; see \cite{DKT}.
%Every ideal of the form $\fm J$ for an ideal $J$, integrally closed ideal, and $\fm$-primary weakly $\fm$-full ideal is Burch; see \cite{DKT} for details.
An ideal $I$ is Burch if it is of the form $\fm J$, integrally closed, or $\fm$-primary weakly $\fm$-full ideal; see \cite{DKT} for details.
It has been shown in \cite[Theorem 5(ii)]{Bur} that every Burch ideal $I$ is the first syzygy of a $2$-Tor-rigid module $R/I$.

\begin{cor}
Let $R$ be a noetherian local ring satisfying $(S_1)$, $I$ a Burch ideal, and $N$ a finitely generated $R$-module.
Assume: 
\begin{enumerate}[\rm(a)]
\item
$N$ satisfies $(S_1)$ and $\cid N_\fp < \infty$ for $\fp \in \sX^1(R)$.
\item
$\height I \ge 2$.
%\item
%$I \otimes_R N$ satisfies $(S_2)$. 
\end{enumerate}
If $I \otimes_R N$ satisfies $(S_2)$, then $\pd_R N \le 2$.
\end{cor}

\begin{proof}
We apply Theorem \ref{2srt} for $R$-modules $N$ and $M:= I$.
Since $\height I \ge 2$, $I$ is locally free on $\sX^1(R)$ and hence the condition (3) in Theorem \ref{2srt} is satisfied. 
The condition (2) follows from Example \ref{ex}(3).
Thus, Theorem \ref{2srt} shows that $\Tor_{i+1}^R(R/I, N) \cong \Tor_i^R(I, N) = 0$ for $i \ge 1$.
Again using \cite[Theorem 5(ii)]{Bur}, we conclude that $\pd_R N\le 2$.
\end{proof}

\begin{ac}
The author is grateful to Olgur Celikbas for his helpful comments during the preparation of this paper.	
\end{ac}

\end{document}